\documentclass[12pt]{amsart}

\usepackage[utf8]{inputenc}

\usepackage[english]{babel}

\usepackage{amsmath,amsfonts,amssymb,amsthm}

\usepackage{amscd}

\usepackage[a4paper,nomarginpar]{geometry}

\usepackage{epsfig,graphicx,color}

\usepackage[all]{xy}

\newtheorem{theorem}{Theorem}
\newtheorem{lemma}[theorem]{Lemma}
\newtheorem{proposition}[theorem]{Proposition}

\theoremstyle{definition}

\newtheorem{remark}[theorem]{Remark}
\newtheorem{question}[theorem]{Question}

\newcommand{\N}{\mathbb{N}}  % set of natural numbers
\newcommand{\Z}{\mathbb{Z}}  % set of integer numbers
\newcommand{\R}{\mathbb{R}}  % set of real numbers
\newcommand{\C}{\mathbb{C}}  % set of complex numbers
\newcommand{\D}{\mathbb{D}}  % unit disc
\newcommand{\T}{\mathbb{T}}  % unit circle
\newcommand{\K}{\mathbb{K}}  % field

\newcommand{\eps}{\varepsilon} % abbreviation for epsilon
\newcommand{\ov} {\overline}   % abbreviation for overline
\newcommand{\cB}{\mathcal{B}}
\newcommand{\cC}{\mathcal{C}}

\newcommand{\Lip}{\operatorname{Lip}}

\allowdisplaybreaks

\begin{document}


\title[Shadowing and Structural Stability in Linear Dynamical Systems]
      {Shadowing and Structural Stability\\ in Linear Dynamical Systems}

% Only \author and \address are required; other information is optional.
% Remove any unused author tags.

\author{Nilson C. Bernardes Jr.}
\address{Departamento de Matem\'atica Aplicada, Instituto de Matem\'atica,
    Universidade Federal do Rio de Janeiro, Caixa Postal 68530,
    Rio de Janeiro, RJ, 21945-970, Brazil.}
\curraddr{}
\email{ncbernardesjr@gmail.com}
\thanks{The first author was partially supported by grant
        {\#}2017/22588-0, S\~ao Paulo Research Foundation (FAPESP)}

\author{Ali Messaoudi}
\address{Departamento de Matem\'atica, Universidade Estadual Paulista,
    Rua Crist\'ov\~ao Colombo, 2265, Jardim Nazareth, S\~ao Jos\'e do
    Rio Preto, SP, 15054-000, Brasil.}
\curraddr{}
\email{ali.messaoudi@unesp.br}
\thanks{The second author was partially supported by project
{\#}307776/2015-8 of CNPq and by projects {\#}2013/24541-0 and
{\#}2017/22588-0 of FAPESP}



\date{} %\today

\dedicatory{}

\begin{abstract}
A well-known result in the area of dynamical systems asserts that any
invertible hyperbolic operator on any Banach space is structurally stable.
This result was originally obtained by P. Hartman in 1960 for operators
on finite-dimensional spaces. The general case was independently obtained
by J. Palis and C. Pugh around 1968.
We will exhibit examples of structurally stable operators that are not
hyperbolic, thereby showing that the converse of the above-mentioned
result is false in general.
We will also prove that an invertible operator on a Banach space is
hyperbolic if and only if it is expansive and has the shadowing property.
Moreover, we will show that if a structurally stable operator is expansive,
then it must be uniformly expansive.
Finally, we will characterize the weighted shifts on the spaces $c_0(\Z)$
and $\ell_p(\Z)$ ($1 \leq p < \infty$) that satisfy the shadowing property.
\end{abstract}

\maketitle

%%%%%%%%%%%%%%%%%%%%%%%%%%%%%%%%%%%%%%%%%%%%%%%%%%%%%%%%%%%%%%%%%%%%%%%%%%%%%

\section{Introduction}

Given a metric space $M$ with metric $d$ and a homeomorphism $h : M \to M$,
recall that a sequence $(x_n)_{n \in \Z}$ in $M$ is called a
{\em $\delta$-pseudotrajectory} of $h$, where $\delta > 0$, if
$$
d(h(x_n),x_{n+1}) \leq \delta \ \text{ for all } n \in \Z.
$$
The homeomorphism $h$ is said to have the {\em shadowing property} if
for every $\eps > 0$ there exists $\delta > 0$ such that every
$\delta$-pseudotrajectory $(x_n)_{n \in \Z}$ of $h$ is {\em $\eps$-shadowed}
by a real trajectory of $h$, that is, there exists $x \in M$ such that
$$
d(x_n,h^n(x)) < \eps \ \text{ for all } n \in \Z.
$$
This notion, which comes from the works of
Sina$\breve{\text{\i}}$~\cite{JSin72} and Bowen~\cite{RBow75},
plays a fundamental role in several branches of the area of dynamical
systems. We refer the reader to the books \cite{AKatBHas95,KPal00,SPil99}
for nice expositions of the shadowing property and its applications.

\smallskip
Another fundamental notion in the area of dynamical systems is that of
{\em hyperbolicity}. For (continuous linear) operators $T$ on complex
Banach spaces, this notion can be defined by requiring that the spectrum
$\sigma(T)$ of $T$ does not intersect the unit circle $\T$. In the case of
real Banach spaces, it is required that $\sigma(T_\C) \cap \T = \emptyset$,
where $T_\C$ denotes the complexification of $T$.
For invertible operators on Banach spaces, it was well-known that
hyperbolicity implies the shadowing property and that the converse holds
on finite-dimensional spaces \cite{JOmb94} and for normal operators on
Hilbert spaces \cite{MMaz00}. It remained open for a while whether or not
this converse was always true. This problem was finally solved by Bernardes
et al. \cite{BerCirDarMesPuj18}, where examples of nonhyperbolic
invertible operators with the shadowing property were obtained.
In view of this fact, it is natural to ask which additional condition
one has to add to shadowing in order to get hyperbolicity. The answer
is {\em expansivity}. Indeed, we will establish the following result
in Section~\ref{HES}:
{\it An invertible operator $T$ on a Banach space $X$ is hyperbolic
if and only if it is expansive and has the shadowing property}.

\smallskip
Yet another fundamental notion in the area of dynamical systems is that
of {\em structural stability}, which was introduced by Andronov and
Pontrjagin \cite{AAndLPon37}. Actually, there are many variations of this
notion. Here we will follow Pugh~\cite{CPug69}. Given a Banach space $X$,
we denote by $C_b(X)$ the Banach space of all bounded continuous maps
$\varphi : X \to X$ endowed with the supremum norm
$$
\|\varphi\|_\infty = \sup_{x \in X} \|\varphi(x)\|.
$$
Recall that a map $\varphi : X \to X$ is said to be {\em Lipschitz}
if the constant
$$
\Lip(\varphi) = \sup_{x \neq y}
\frac{\|\varphi(x)-\varphi(y)\|}{\|x-y\|}
$$
is finite. Recall also that two maps $\varphi : X \to X$ and $\psi : X \to X$
are {\em topologically conjugate} if there is a homeomorphism
$h : X \to X$ such that $h \circ \varphi = \psi \circ h$. An invertible
operator $T$ on $X$ is said to be {\em structurally stable} if there exists
$\eps > 0$ such that $T + \varphi$ is topologically conjugate to $T$
whenever $\varphi \in C_b(X)$ is a Lipschitz map with
$\|\varphi\|_\infty \leq \eps$ and $\Lip(\varphi) \leq \eps$.
Moreover, $T$ is said to be
{\em structurally stable relative to} $GL(X)$ (the group of all invertible
operators on $X$) if there exists $\eps > 0$ such that $S$ is topologically
conjugate to $T$ whenever $S \in GL(X)$ and $\|S - T\| \leq \eps$ (here
$\|\cdot\|$ denotes the operator norm). The following result is classical:
{\it Every invertible hyperbolic operator $T$ on a Banach space $X$ is:
\begin{itemize}
\item [\rm (a)] structurally stable;
\item [\rm (b)] structurally stable relative to $GL(X)$.
\end{itemize}}
\noindent Conclusion (a) is often called Hartman's theorem. The original
version was obtained by Hartman \cite{PHar60} in the finite-dimensional
setting. The general case was independently obtained by Palis \cite{JPal68}
and Pugh \cite{CPug69}, motivated by an argument in Moser \cite{JMos69}.

\smallskip
We will prove in Section~\ref{HES} that there is a class $\cC$ of
invertible operators on certain Banach spaces $X$ such that each operator
in $\cC$ is structurally stable, but it is not hyperbolic and it is not
structurally stable relative to $GL(X)$.
This shows that the notion of structural stability for operators
does not coincide with the notion of hyperbolicity and does not coincide
with structural stability relative to $GL(X)$. It seems to be still an open
problem whether or not hyperbolicity and structural stability relative to
$GL(X)$ coincide. Nevertheless, it is well-known that these three notions
coincide in the finite-dimensional setting (see \cite[Theorem~2.4]{JRob72}).

\smallskip
Surprisingly enough, the structurally stable operators in the class $\cC$
have nontrivial fixed points and satisfy the so-called {\em frequent
hypercyclicity criterion} \cite[Section~9.2]{KGroAPer11}, which is a very
strong notion of chaos in linear dynamics. It simultaneously implies
frequent hypercyclicity, mixing, Devaney chaos, dense distributional chaos
and dense mean Li-Yorke chaos \cite{BerBonMulPer13,BerBonPer,KGroAPer11}.

\smallskip
We will also prove the following result in Section~\ref{HES}:
{\it Suppose that an invertible operator $T$ on a complex Banach space $X$
is structurally stable or structurally stable relative to $GL(X)$.
Then, the following facts hold:
\begin{itemize}
\item [\rm (a)] If $T$ is expansive, then $T$ is uniformly expansive.
\item [\rm (b)] If $T$ is positively expansive, then $T$ is hyperbolic.
\end{itemize}}

\noindent
Moreover, we will show that if an invertible weighted shift on a space
$\ell_p(\Z)$ ($1 \leq p < \infty$) or $c_0(\Z)$ is expansive and
strongly structurally stable, then it must be hyperbolic.

\smallskip
Due to the importance of {\em weighted shifts} in the area of operator theory
and its applications, many researchers have extensively studied the dynamics
of this class of operators (see the books \cite{FBayEMat09,KGroAPer11} and
the papers \cite{BarMarPer12,FBayIRuz15,BerBonMulPer13,BerBonMulPer15,
BerCirDarMesPuj18,GCosMSam04,KGro00,HSal95}, for instance).
In Section~\ref{Shifts} we will contribute to this line of investigation by
establishing the following characterization of the weighted shifts that
satisfy the shadowing property:
{\it Let $X = \ell_p(\Z)$ $(1 \leq p < \infty)$ or $X = c_0(\Z)$,
let $w = (w_n)_{n \in \Z}$ be a bounded sequence of scalars with
$\inf_{n \in \Z} |w_n| > 0$, and consider the bilateral weighted backward
shift
$$
B_w : (x_n)_{n \in \Z} \in X \mapsto (w_{n+1}x_{n+1})_{n \in \Z} \in X.
$$
Then $B_w$ has the shadowing property if and only if one of the
following conditions holds:
\begin{itemize}
\item [\rm (A)] $\displaystyle \lim_{n \to \infty} \sup_{k \in \Z}
                |w_k w_{k+1} \cdot\ldots\cdot w_{k+n}|^\frac{1}{n} < 1$.
\item [\rm (B)] $\displaystyle \lim_{n \to \infty} \inf_{k \in \Z}
                |w_k w_{k+1} \cdot\ldots\cdot w_{k+n}|^\frac{1}{n} > 1$.
\item [\rm (C)] $\displaystyle \lim_{n \to \infty} \sup_{k \in \N}
                |w_{-k} w_{-k-1} \cdot \ldots \cdot w_{-k-n}|^\frac{1}{n} < 1$
                and
                $\displaystyle \lim_{n \to \infty} \inf_{k \in \N}
                |w_k w_{k+1} \cdot \ldots \cdot w_{k+n}|^\frac{1}{n} > 1$.
\end{itemize}}
\noindent We observe that (A) and (B) are exactly the cases in which the
operator $B_w$ is hyperbolic \cite[Remark~35]{BerCirDarMesPuj18}.
In these cases, $B_w$ does not exhibit chaotic behaviour, since hyperbolic
operators are not even Li-Yorke chaotic \cite[Theorem~C]{BerCirDarMesPuj18},
which is a very weak notion of chaos. On the other hand, we will see
later that condition (C) implies the frequent hypercyclicity criterion,
and so, in this case, $B_w$ exhibit several types of chaotic behaviours.

%%%%%%%%%%%%%%%%%%%%%%%%%%%%%%%%%%%%%%%%%%%%%%%%%%%%%%%%%%%%%%%%%%%%%%%%%%%%%

\section{Preliminaries}

The letter $\N$ denotes the set of all positive integers and
$\N_0 = \N \cup \{0\}$. Moreover, $\D$ and $\T$ denote the open unit disc
and the unit circle in the complex plane $\C$, respectively.

\smallskip
Given a Banach space $X$, we denote by $B_X$ and by $S_X$ the closed unit
ball and the unit sphere of $X$, respectively. By an {\em operator} on $X$
we mean a bounded linear map $T$ form $X$ into itself. $\cB(X)$ denotes
the space of all operators on $X$ with the operator norm. If $T \in \cB(X)$,
then $\sigma(T)$, $\sigma_p(T)$ and $\sigma_a(T)$ denote the spectrum,
the point spectrum and the approximate point spectrum of $T$.
Recall that $\sigma_p(T)$ is the set of all eigenvalues of $T$ and that
$\lambda \in \sigma_a(T)$ if and only if there is a sequence $(x_n)_{n\in\N}$
in $S_X$ such that $\|Tx_n - \lambda x_n\| \to 0$ as $n \to \infty$.
Recall also that $\sigma_p(T) \subset \sigma_a(T) \subset \sigma(T)$
and that $\sigma_a(T)$ contains the boundary of $\sigma(T)$, which is
a compact set. We denote by $r(T)$ the spectral radius of $T$.
Recall that the spectral radius formula asserts that
$r(T) = \lim_{n \to \infty} \|T^n\|^\frac{1}{n}$
(see \cite{Y80}, for instance).

\smallskip
Let $c_0(\Z)$ be the Banach space of all sequences $(x_n)_{n \in \Z}$
of scalars such that $\lim_{|n| \to \infty} x_n = 0$ endowed with the
supremum norm. Let $\ell_p(\Z)$ be the Banach space of all absolutely
$p$-summable sequences $(x_n)_{n \in \Z}$ of scalars endowed with the
$p$-norm, where $1 \leq p < \infty$. Moreover, let $(e_n)_{n \in \Z}$
denote the sequence of canonical unit vectors in $c_0(\Z)$ and in
$\ell_p(\Z)$.

\smallskip
Recall that a continuous map $f : M \to M$, where $M$ is a metric space,
is said to be {\em expansive} if there exists a constant $e > 0$ such that
for every pair $x,y$ of distinct points in $M$, there exists $n \in \Z$
with $d(f^n(x),f^n(y)) \geq e$. In the case $T$ is an invertible operator
on a Banach space $X$, this is equivalent to say that there exists a
constant $c > 1$ such that for each point $x \in S_X$, there exists
$n \in \Z$ with $\|T^nx\| \geq c$. Recall also that $T$ is said to be
{\em uniformly expansive} if there exist $c > 1$ and $m \in \N$ such that
$$
x \in S_X \ \ \Longrightarrow \ \ \|T^mx\| \geq c \ \text{ or } \
                                  \|T^{-m}x\| \geq c.
$$
Actually, in the notions of expansivity and uniform expansivity for
operators, we can choose any $c > 1$ we want.
It is known that $T$ is uniformly expansive if and only if
$\sigma_a(T) \cap \T = \emptyset$ \cite{MEisJHed70,JHed71}.
In particular, every invertible hyperbolic operator is uniformly expansive.

\smallskip
The following fact will be used a few times: {\it if $T$ is an invertible
operator on a Banach space $X$ and $\varphi \in C_b(X)$ is a Lipschitz map
with $\|\varphi\|_\infty$ and $\Lip(\varphi)$ sufficiently small, then
$T + \varphi$ is a homeomorphism} (see \cite[Lemma~1]{CPug69}, for instance).

%%%%%%%%%%%%%%%%%%%%%%%%%%%%%%%%%%%%%%%%%%%%%%%%%%%%%%%%%%%%%%%%%%%%%%%%%%%%%

\section{Hyperbolicity, shadowing and structural stability}\label{HES}

We begin this section by recalling some notions from topological dynamics.
Let $M$ be a metric space with metric $d$ and $h : M \to M$ a homeomorphism.
For each $x \in M$ and $\eps > 0$, we define the {\em $\eps$-stable set}, the
{\em $\eps$-unstable set}, the {\em stable set} and the {\em unstable set
of $h$ at $x$} by
\begin{align*}
W^s_\eps(x,h) &= \{y \in M : d(h^n(x),h^n(y)) \leq \eps
                            \text{ for all } n \geq 0\},\\
W^u_\eps(x,h) &= \{y \in M : d(h^{-n}(x),h^{-n}(y)) \leq \eps
                            \text{ for all } n \geq 0\},\\
W^s(x,h) &= \{y \in M : d(h^n(x),h^n(y)) \to 0 \text{ as } n \to \infty\},\\
W^u(x,h) &= \{y \in M : d(h^{-n}(x),h^{-n}(y)) \to 0 \text{ as } n
\to \infty\},
\end{align*}
respectively. The homeomorphism $h$ is said to have {\em canonical
coordinates} if for each $\eps > 0$ there exists $\delta > 0$ such
that $d(x,y) \leq \delta$ implies $W^s_\eps(x,h) \cap
W^u_\eps(y,h)\neq \emptyset$. Moreover, we say that $h$ has {\em
hyperbolic coordinates} if $h$ has canonical coordinates and there
exist constants $c \geq 1$, $0 < \beta < 1$ and $\gamma > 0$ such
that the following properties hold for each $x \in M$:
\begin{itemize}
\item $y \in W^s_\gamma(x,h) \ \Longrightarrow \
      d(h^n(x),h^n(y)) \leq c \beta^n d(x,y)$ for all $n \in \N$.
\item $y \in W^u_\gamma(x,h) \ \Longrightarrow \
      d(h^{-n}(x),h^{-n}(y)) \leq c \beta^n d(x,y)$ for all $n \in \N$.
\end{itemize}
The original notion of ``canonical coordinates'' came from Smale's
investigations on Axiom~A diffeomorphisms~\cite{SSma67}. Here we are
following the terminology of~\cite{NAok89}. We will also consider
the sets
\begin{align*}
W^{b,s}(x,h) &= \big\{y \in M : \big(d(h^n(x),h^n(y))\big)_{n \in
\N}
                \text{ is bounded}\big\},\\
W^{b,u}(x,h) &= \big\{y \in M : \big(d(h^{-n}(x),h^{-n}(y))\big)_{n
\in \N}
                \text{ is bounded}\big\}.
\end{align*}

In the case $T$ is an invertible operator on a Banach space $X$, we
have that
$$
W^s_\eps(x,T) = x + W^s_\eps(0,T), \ W^s(x,T) = x + W^s(0,T),
\ W^{b,s}(x,T) = x + W^{b,s}(0,T),
$$
and similarly with ``$u$'' instead of ``$s$''. Moreover,
$W^{b,s}(0,T)$ and $W^{b,u}(0,T)$ are vector subspaces of $X$
satisfying
$$
T(W^{b,s}(0,T)) = W^{b,s}(0,T) \ \ \text{ and } \ \ T(W^{b,u}(0,T))
= W^{b,u}(0,T).
$$

\begin{theorem}\label{HyperExpShad}
For any invertible operator $T$ on any Banach space $X$,
the following assertions are equivalent:
\begin{itemize}
\item [\rm   (i)] $T$ is hyperbolic;
\item [\rm  (ii)] $T$ has hyperbolic coordinates;
\item [\rm (iii)] $T$ is expansive and has the shadowing property.
\end{itemize}
\end{theorem}

For the proof we will need the following result.

\begin{proposition}\label{UnifExp}
Let $T$ be an invertible operator on a Banach space $X$. If $T$ is
uniformly expansive, then there are constants $c \geq 1$ and $0 <
\beta < 1$ such that
\begin{equation}\label{wbs}
W^{b,s}(0,T) = W^s(0,T) = \{x \in X : \|T^nx\| \leq c \beta^n \|x\|
                                      \text{ for all } n \in \N\}
\end{equation}
and
\begin{equation}\label{wbu}
W^{b,u}(0,T) = W^u(0,T) = \{x \in X : \|T^{-n}x\| \leq c \beta^n
\|x\|
                                      \text{ for all } n \in \N\}.
\end{equation}
In particular, the subspaces $W^{b,s}(0,T)$ and $W^{b,u}(0,T)$ are
closed in $X$.
\end{proposition}

\begin{proof}
Since $T$ is uniformly expansive, there exists $m \in \N$ such that
$S_X = A \cup B$, where
$$
A = \{x \in S_X : \|T^mx\| \geq 2\} \ \ \text{ and } \ \ B = \{x \in
S_X : \|T^{-m}x\| \geq 2\}.
$$
Note that
\begin{equation}\label{uea}
x \in A \ \Longrightarrow \ \frac{T^mx}{\|T^mx\|} \in A,
\end{equation}
for otherwise we would get $1 = \|x\| \geq 2 \|T^mx\| \geq 4$,
a contradiction. Similarly,
\begin{equation}\label{ueb}
x \in B \ \Longrightarrow \ \frac{T^{-m}x}{\|T^{-m}x\|} \in B.
\end{equation}
Let
$$
\beta = \Big(\frac{1}{2}\Big)^\frac{1}{m} \ \ \text{ and } \ \ c =
\max\Big\{\frac{\|T^j\|}{\beta^{|j|}} : -m < j < m\Big\}.
$$
Let $W$ denote the set in the right-hand side of (\ref{wbs}). It is
clear that $W \subset W^s(0,T) \subset W^{b,s}(0,T)$. Take $x \in
W^{b,s}(0,T) \backslash \{0\}$. Given $n \in \N$, we define two sequences
$(y_k)_{k \in \N}$ and $(z_k)_{k \in \N}$ as follows: $y_1 = z_1 =
\frac{T^nx}{\|T^nx\|}$, $y_k = \frac{T^my_{k-1}}{\|T^my_{k-1}\|}$
and $z_k = \frac{T^{-m}z_{k-1}}{\|T^{-m}z_{k-1}\|}$ ($k \geq 2$).
Note that, for every $k \geq 2$,
$$
y_k = \frac{T^{(k-1)m+n}x}
           {\|T^my_1\| \cdot\ldots\cdot \|T^my_{k-1}\|\|T^nx\|}
\ \text{ and } \ z_k = \frac{T^{-(k-1)m+n}x}
           {\|T^{-m}z_1\| \cdot\ldots\cdot \|T^{-m}z_{k-1}\|\|T^nx\|}\cdot
$$
If $y_1 \in A$, then (\ref{uea}) implies that $y_k \in A$ for all $k
\in \N$, and so
$$
\|T^{km+n}x\| = \|T^my_1\|\cdot\ldots\cdot\|T^my_k\| \cdot \|T^nx\|
  \geq 2^k \|T^nx\| \ \ \ \ (k \in \N).
$$
This contradicts the fact that $(T^jx)_{j \in \N}$ is bounded. Thus,
we must have $z_1 = y_1 \in B$. By (\ref{ueb}), $z_k \in B$ for all
$k \in \N$, and so
$$
\|T^{-km+n}x\| = \|T^{-m}z_1\|\cdot\ldots\cdot\|T^{-m}z_k\| \cdot
\|T^nx\|
  \geq 2^k \|T^nx\| \ \ \ \ (k \in \N).
$$
Since $n \in \N$ is arbitrary, we may choose $n = km$ and obtain
$$
\|T^{km}x\| \leq \frac{1}{2^k} \|x\| = \beta^{km} \|x\| \ \ \ \ (k
\in \N).
$$
Now, an easy computation shows that $x \in W$. This proves
(\ref{wbs}). By applying (\ref{wbs}) with $T^{-1}$ instead of $T$ we
obtain (\ref{wbu}).
\end{proof}

\begin{remark}
We cannot replace the hypothesis of uniform expansivity by expansivity in
Proposition~\ref{UnifExp}. For instance, let $X = \ell_p(\Z)$
($1 \leq p < \infty$) or $X = c_0(\Z)$, and let
$$
B_w : (x_n)_{n \in \Z} \in X \mapsto (w_{n+1}x_{n+1})_{n \in \Z} \in
X
$$
be the bilateral weighted backward shift on $X$ whose weight sequence
$w = (w_n)_{n \in \Z}$ is given by $w_n = 2$ if $n < 0$, $w_n = 1$ if
$n \geq 0$. Then $B_w$ is positively expansive (for each $x \in S_X$,
there exists $n \in \N$ such that $\|B_w^nx\| \geq 2$) and, in particular,
it is expansive. However, $W^{b,u}(0,B_w) = X$ and $W^u(0,B_w) = \{0\}$.
\end{remark}

\begin{remark}
It is not difficult to prove that if $T \in \cB(X)$ is uniformly expansive,
then for all $ \eps >0$ and $x \in X$, we have
$$
W^s(x,T) = \bigcup_{n=0}^{\infty} T^{-n}\big(W_{\eps}^s(T^n x,T)\big)
\ \text{ and } \
W^u(x,T) = \bigcup_{n=0}^{\infty} T^{n}\big(W_{\eps}^u(T^{-n} x,T)\big).
$$
\end{remark}

\smallskip
\noindent {\it Proof of Theorem \ref{HyperExpShad}.}
(i) $\Rightarrow$ (iii): It is well-known that every invertible
hyperbolic operator on a Banach space is expansive (even uniformly
expansive) \cite{MEisJHed70} and has the shadowing property
\cite{JOmb94}.

\smallskip
\noindent (iii) $\Rightarrow$ (ii): Let us show that $T$ has
canonical coordinates. Fix $\eps > 0$ arbitrary. Since $T$ has the
shadowing property, there exists $\delta \in (0,\frac{\eps}{2})$
such that every $\delta$-pseudotra\-jectory of $T$ is
$\frac{\eps}{2}$-shadowed by a real trajectory of $T$. Given $x,y
\in X$ with $\|x-y\| \leq \delta$, we define
$$
z_n = \left\{\begin{array}{cl}
             T^nx & \text{if } n \geq 0\\
             T^ny & \text{if } n < 0.
             \end{array}
             \right.
$$
By shadowing, there exists a point $z \in X$ such that
$$
\|z_n - T^nz\| < \eps/2 \ \ \text{ for all } n \in \Z.
$$
By the definition of $(z_n)$, we see that $z \in W^s_\eps(x,T) \cap
W^u_\eps(y,T)$. Now, if $T$ is not uniformly expansive, then there exists
$\lambda \in \sigma_a(T) \cap \T$ \cite{JHed71} and we can get a point
$x_0 \in S_X$ such that $\|Tx_0 - \lambda x_0\| \leq \delta$.
Since $(\lambda^n x_0)_{n \in \Z}$ is a $\delta$-pseudotrajectory of $T$,
there exists $y \in X$ such that $\|\lambda^n x_0 - T^n y\| < \eps$ for all
$n \in \Z$. This implies that $1 - \eps < \|T^n y\| < 1 + \eps$ for all
$n \in \Z$, which contradicts the hypothesis that $T$ is expansive.
Thus, $T$ is uniformly expansive and it follows from
Proposition~\ref{UnifExp} that $T$ has hyperbolic coordinates.

\smallskip
\noindent (ii) $\Rightarrow$ (i): Let $E = W^{b,s}(0,T)$ and $F =
W^{b,u}(0,T)$. We know that $E$ and $F$ are vector subspaces of $X$
satisfying $T(E) = E$ and $T(F) = F$. By hypothesis, there exist
constants $c \geq 1$, $0 < \beta < 1$ and $\gamma > 0$ such that:
\begin{itemize}
\item $x \in W^s_\gamma(0,T) \ \Longrightarrow \ \|T^nx\| \leq c \beta^n \|x\|$
      for all $n \in \N$.
\item $x \in W^u_\gamma(0,T) \ \Longrightarrow \ \|T^{-n}x\| \leq c \beta^n \|x\|$
      for all $n \in \N$.
\end{itemize}
This implies that
\begin{equation}\label{eqm}
E = \{x \in X : \|T^nx\| \leq c \beta^n \|x\| \text{ for all } n \in
\N\}
\end{equation}
and
\begin{equation}\label{eqn}
F = \{x \in X : \|T^{-n}x\| \leq c \beta^n \|x\| \text{ for all } n
\in \N\}.
\end{equation}
In particular, $E$ and $F$ are closed in $X$. Moreover, if $x \in E
\cap F$ then
$$
\|x\| = \|T^nT^{-n}x\| \leq c \beta^n \|T^{-n}x\| \leq c^2
\beta^{2n} \|x\|,
$$
for all $n \in \N$, and so $x = 0$. This proves that $E \cap F =
\{0\}$. Now, let us prove that $E + F = X$. Since $T$ has canonical
coordinates, there exists $\delta > 0$ such that $W^s_1(x,T) \cap
W^u_1(y,T) \neq \emptyset$ whenever $\|x-y\| \leq \delta$. It is
enough to show that the ball $B_X(0;\delta)$ is contained in $E+F$.
So, take a point $x$ in this ball. Then, there must exist a point $z
\in W^s_1(x,T) \cap W^u_1(0,T)$. This implies that $x-z \in E$ and
$z \in F$. And, of course, $x = (x-z)+z$. Finally, it follows from
(\ref{eqm}) that $\|(T|_E)^n\| \leq c \beta^n$ for all $n \in \N$.
Hence, the spectral radius formula gives $r(T|_E) \leq \beta < 1$.
Analogously, by (\ref{eqn}), $r(T^{-1}|_F) < 1$.
This completes the proof that $T$ is hyperbolic.
\hfill $\Box$

\medskip
For not necessarily invertible maps, we can define the notion of
{\em positive shadowing} simply by replacing the set $\Z$ by the set
$\N_0$ in the definition of shadowing.

\begin{proposition}
For any operator $T$ on any Banach space $X$, the following assertions
are equivalent:
\begin{itemize}
\item [\rm  (i)] $T$ is hyperbolic with spectrum contained in
                 $\C \backslash \ov{\D}$;
\item [\rm (ii)] $T$ is positively expansive and has the positive shadowing
                 property.
\end{itemize}
\end{proposition}

\begin{proof}
(i) $\Rightarrow$ (ii): Clear.

\smallskip
\noindent (ii) $\Rightarrow$ (i): Suppose that $T$ is not invertible.
Since $T$ is positively expansive, then $T$ is injective and not surjective.
Let $\delta > 0$ be associated to $\eps = 1$ in the definition of positive
shadowing. Fix $z \in X \backslash T(X)$ with $\|z\| = \delta$.
Consider the $\delta$-pseudotrajectory $(y_n)_{n \in \N_0}$ of $T$ defined by
$$
y_0 = 0, \ \ y_1 = z \ \ \text{ and } \ \ y_n = Ty_{n-1} = T^{n-1} z \
\text{ for } n \geq 2.
$$
By positive shadowing, there exists $y \in X$ such that
$$
\|T^{n-1}(Ty - z)\| < 1 \ \ \text{ for all } n \in \N.
$$
Since $T$ is positively expansive, we must have $Ty = z$.
This is a contradiction, since $z$ does not belong to $T(X)$.
Thus, the operator $T$ is invertible.

Now, we claim that
\begin{equation}\label{ab}
\sigma_{a}(T) \subset \C \backslash \ov{\D}.
\end{equation}
Indeed, suppose that this is false and take a
$\lambda \in \sigma_a(T) \cap \ov{\D}$. Let $x_0 \in S_X$ be such that
$\|Tx_0 - \lambda x_0\| \leq \delta$. Since $(\lambda^n x_0)_{n \in \N_0}$
is a $\delta$-pseudotrajectory of $T$, there exists $x \in X$ such that
$\|\lambda^n x_0 - T^n x\| < 1$ for all $n \in \N_0$. This implies that
$x \neq 0$ and $\|T^nx\| < 2$ for all $n \in \N$, contradicting the
hypothesis that $T$ is positively expansive. Thus, (\ref{ab}) holds.

Finally, since $0 \not\in \sigma(T)$ and $\sigma_a(T)$ contains the
boundary of $\sigma(T)$, (\ref{ab}) implies that
$\sigma(T) \subset \C \backslash \ov{\D}$.
\end{proof}   

\begin{theorem}\label{HyperSS}
Suppose that an invertible operator $T$ on a complex Banach space $X$
is structurally stable or structurally stable relative to $GL(X)$.
Then, the following facts hold:
\begin{itemize}
\item [\rm (a)] If $T$ is expansive, then $T$ is uniformly expansive.
\item [\rm (b)] If $T$ is positively expansive, then $T$ is hyperbolic.
\end{itemize}
\end{theorem}

For the proof we will need the following simple fact.

\begin{lemma}\label{Extension}
If $\alpha : X \to X$ is a Lipschitz map, then there exists a bounded
Lipschitz map $\varphi : X \to X$ which extends $\alpha|_{B_X}$ and
satisfies $\|\varphi\|_\infty \leq \|\alpha|_{2B_X}\|_\infty$ and
$\Lip(\varphi) \leq 3 \Lip(\alpha)$.
\end{lemma}

Many variations of this lemma have appeared elsewhere, but we shall present
a proof below for the sake of completeness.

\begin{proof}
Let $\rho : \R \to [0,1]$ be given by $\rho = 1$ on $(-\infty,1]$,
$\rho(t) = 2-t$ for $t \in [1,2]$, and $\rho = 0$ on $[2,+\infty)$.
Define
$$
\varphi(x) = \alpha(0) + \rho(\|x\|) (\alpha(x) - \alpha(0)) \ \ \ \ \ (x \in X).
$$
Clearly, $\varphi = \alpha$ on $B_X$ and
$\|\varphi\|_\infty \leq \|\alpha|_{2B_X}\|_\infty$.
Let $x,y \in X$. If $x \in 2B_X$ then
\begin{align*}
\|\varphi(x) - \varphi(y)\|
&= \big\|\big(\rho(\|x\|) - \rho(\|y\|)\big)\big(\alpha(x) - \alpha(0)\big)
   + \rho(\|y\|) \big(\alpha(x)-\alpha(y)\big)\big\|\\
&\leq 3 \Lip(\alpha) \|x-y\|.
\end{align*}
The same holds if $y \in 2B_X$ and the inequality is trivial if both
$x$ and $y$ lie outside $2B_X$.
\end{proof}

\noindent
{\it Proof of Theorem \ref{HyperSS}.}
(a): Assume that $T$ is expansive but it is not uniformly expansive.
Then, there exists $\lambda \in \sigma_a(T) \cap \T$ \cite{JHed71}.
Given any $\eps > 0$, there exists $y \in S_X$ such that
$\|Ty - \lambda y\| \leq \eps$. By the Hahn-Banach extension theorem,
there is a linear functional $\psi$ on $X$ such that $\|\psi\| = 1$ and
$\psi(ay) = a$ for all $a \in \C$. Let $S \in \cB(X)$ be defined by
$$
Sx = Tx + \psi(x)(\lambda y - Ty).
$$
Then $\|S - T\| \leq \eps$ and $Sy = \lambda y$.

In the case $T$ is structurally stable relative to $GL(X)$, by choosing
$\eps > 0$ small enough, there is a homeomorphism $h : X \to X$ such that
$T \circ h = h \circ S$. Hence,
$$
T^n(h(y)) = h(\lambda^n y) \ \ \text{ for all } n \in \Z.
$$
Since the sequence $(h(\lambda^n y))_{n \in \Z}$ is bounded
(because it is contained in the compact set $\{h(\beta y) : \beta \in \T\}$)
and $h(y) \neq 0$, this contradicts the expansivity of $T$.

In the case $T$ is structurally stable, we apply Lemma~\ref{Extension}
to obtain a Lipschitz map $\varphi : X \to X$ such that $\varphi = S - T$
on $B_X$, $\|\varphi\|_\infty \leq 2\eps$ and $\Lip(\varphi) \leq 3\eps$.
By choosing $\eps > 0$ small enough, there is a homeomorphism $h : X \to X$
such that $T \circ h = h \circ (T + \varphi)$.
Since $T + \varphi = S$ on $B_X$, we have that $(T + \varphi)^{-1} = S^{-1}$
on $S(B_X)$. Let $\gamma \in (0,1]$ be such that $S(B_X) \supset \gamma B_X$.
Then,
$$
T^n(h(\gamma y)) = h(\lambda^n \gamma y) \ \ \text{ for all } n \in \Z.
$$
As before, this contradicts the expansivity of $T$.

\smallskip
\noindent (b): The proof of (b) is a simple modification of the proof of (a).
Indeed, assume that $T$ is positively expansive but it is not hyperbolic.
Then, $\sigma(T) \cap \T \neq \emptyset$. Since $0 \not\in \sigma(T)$,
the boundary of $\sigma(T)$ must intersect $\ov{\D}$. Thus, there exists
$\lambda \in \sigma_a(T) \cap \ov{\D}$. By arguing as in the proof of~(a),
we obtain a homeomorphism $h : X \to X$ such that
$$
T^n(h(y)) = h(\lambda^n y) \ \ \text{ for all } n \in \N,
$$
no matter if $T$ is structurally stable or structurally stable relative
to $GL(X)$. Since the sequence $(h(\lambda^n y))_{n \in \N}$ is bounded
and $h(y) \neq 0$, we have a contradiction with the hypothesis that $T$
is positively expansive.
\hfill $\Box$

\begin{remark}
For invertible operators $T$ with the property that $\sigma_a(T) = \sigma(T)$,
the above proof shows that if $T$ is expansive and structurally stable
(or structurally stable relative to $GL(X)$), then $T$ is hyperbolic.
In particular, this applies to invertible normal operators on Hilbert spaces.
\end{remark}

A well-known result in the area of dynamical systems, from the 1960's,
asserts that any invertible hyperbolic operator on any Banach space is
structurally stable \cite{PHar60,JPal68,CPug69}. It is also known that
these operators are structurally stable relative to $GL(X)$
\cite{Alves,JRob72} (for the convenience of the reader, we will present
a proof of this fact in an appendix at the end of the paper).
Our goal now is to establish the result below, which shows the existence
of structurally stable operators that are not hyperbolic and are not
structurally stable relative to $GL(X)$. Our examples will be invertible
bilateral weighted backward shifts
$$
B_w : (x_n)_{n \in \Z} \in X \mapsto (w_{n+1}x_{n+1})_{n \in \Z} \in X,
$$
where $X = \ell_p(\Z)$ $(1 \leq p < \infty)$ or $X = c_0(\Z)$.
We recall that $w = (w_n)_{n \in \Z}$ is a bounded sequence of scalars
with $\inf_{n \in \Z} |w_n| > 0$. The boundedness of $w$ is a necessary
and sufficient condition for $B_w$ to be a well-defined operator on $X$
and the condition $\inf_{n \in \Z} |w_n| > 0$ means that $B_w$ is invertible.
We also observe that in the proof of Theorem~\ref{SSShifts} below
we will use the characterization of shadowing for weighted shifts that
will be given in the next section (Theorem~\ref{Characterization3}) and
that has already been stated in the introduction.

\begin{theorem}\label{SSShifts}
Let $X = \ell_p(\Z)$ $(1 \leq p < \infty)$ or $X = c_0(\Z)$.
Every invertible bilateral weighted backward shift $B_w$ on $X$ which has
the shadowing property is structurally stable. In particular, if $w$
satisfies condition (C) of Theorem~\ref{Characterization3}, then $B_w$ is
structurally stable, but it is not hyperbolic and it is not structurally
stable relative to $GL(X)$.
\end{theorem}

Hence, simple concrete examples of nonhyperbolic structurally stable
$B_w$'s can be obtained by putting $w_n = 1/a$ if $n < 0$ and $w_n = a$
if $n \geq 0$, where $a$ is any real number greater than $1$.

\smallskip
The following simple fact concerning hypercyclic operators (i.e., operators
that have a dense orbit) will be used in the proof of the above theorem.

\begin{lemma}\label{Hypercyclic}
For every complex Banach space $X$, the set $HC(X)$ of all hypercyclic
operators on $X$ has empty interior in $\cB(X)$.
\end{lemma}

\begin{proof}
Fix $T \in HC(X)$ and $\eps > 0$. Let $T^{*}$ be the adjoint operator
of $T$ acting on the dual space $X^*$ of $X$ by $T^*x^* = x^* \circ T$.
Take $\lambda \in \sigma_a(T^*)$ and let $y^* \in S_{X^*}$ be such that
$$
\|T^*y^* - \lambda y^*\| \leq \eps/2.
$$
Choose $y \in S_X$ such that $|y^*(y)| \geq 1/2$. Let $R \in \cB(X^*)$ be
defined by
$$
Rx^* = \frac{x^*(y)}{y^*(y)} \cdot (\lambda y^* - T^*y^*).
$$
Since $R$ is $w^*$-$w^*$-continuous, there exists $Q \in \cB(X)$ with
$Q^* = R$ \cite[Theorem~3.1.11]{RMeg98}. Let $S = T+Q \in \cB(X)$.
Then $\|S-T\| = \|Q\| = \|R\| \leq \eps$ and $S^*y^* = \lambda y^*$.
This second property implies that $S \not\in HC(X)$, for otherwise
there would exist $x \in X$ such that the points
$y^*(S^nx) = ((S^*)^ny^*)(x) = \lambda^n y^*(x)$ ($n \in \N_0$)
form a dense set in $\C$, which is impossible.
\end{proof}

\noindent
{\it Proof of Theorem~\ref{SSShifts}.}
We may assume that $B_w$ has the shadowing property but it is not hyperbolic.
Then, by Theorem~\ref{Characterization3},
\begin{equation}\label{ssnh1}
\lim_{n \to \infty} \sup_{k \in \N}
     |w_{-k} w_{-k-1} \cdots w_{-k-n}|^\frac{1}{n} < 1
\ \ \text{ and } \ \ \lim_{n \to \infty} \sup_{k \in \N}
     |w_k w_{k+1} \cdots w_{k+n}|^{-\frac{1}{n}} < 1.
\end{equation}

Our goal is to find a homeomorphism $h: X \to X$ such that
$$
h \circ B_w = S \circ h,
$$
where $S = B_w + \alpha$ with $\alpha \in C_b(X)$ a Lipschitz map such that
$\|\alpha\|_\infty$ and $\Lip(\alpha)$ are sufficiently small. By putting
$h = I + \varphi$, we have that the above equality is equivalent to
$$
\varphi(B_wx) - B_w(\varphi(x)) = \alpha(x + \varphi(x)).
$$
We will first solve the equation 
\begin{equation}\label{ssnh2}
\varphi(B_wx) - B_w(\varphi(x)) = \alpha(x).
\end{equation}
For this purpose, for each $\varphi \in C_b(X)$, we write
$\varphi(x) = \sum_{i \in \Z} \varphi_i(x) e_i$. Let
$$
C_{b,0}(X) = \{\varphi \in C_b(X) : \varphi_0 = 0\},
$$
which is a closed subspace of $C_b(X)$. We claim that the bounded
linear map
$$
F : \varphi \in C_{b,0}(X) \mapsto
    \varphi \circ B_w - B_w \circ \varphi \in C_b(X)
$$
is bijective. Indeed, fix $\alpha \in C_b(X)$.
Equation (\ref{ssnh2}) is equivalent to
$$
\varphi_i(B_wx) - w_{i+1} \varphi_{i+1}(x) = \alpha_i(x) \ \
\text{ for all } i \in \Z,
$$
which implies that
\begin{equation}\label{AAA}
\varphi_i(x) = \frac{\varphi_0(B_w^ix)}{w_1 \cdots w_i}
  - \sum_{j=0}^{i-1} \frac{\alpha_j(B_w^{i-j-1}x)}{w_{j+1} \cdots w_i}
\end{equation}
and
\begin{equation}\label{BBB}
\varphi_{-i}(x) = w_{-i+1} \cdots w_0 \, \varphi_0(B_w^{-i}x)
  + \sum_{j=1}^i w_{-i+1} \cdots w_{-j} \, \alpha_{-j}(B_w^{-i+j-1}x)
\end{equation}
(where in the case $j = i$ we consider the ``empty product''
$w_{-i+1} \cdots w_{-j}$ to be $1$), for every $i \in \N$. Hence,
the only possible solution of (\ref{ssnh2}) in $C_{b,0}(X)$ is given by
$$
\varphi_0 = 0, \ \ \varphi_i(x) = -\sum_{j=0}^{i-1}
\frac{\alpha_j(B_w^{i-j-1}x)}{w_{j+1} \cdots w_i} \ \text{ and } \
\varphi_{-i}(x) = \sum_{j=1}^i w_{-i+1} \cdots w_{-j} \,
\alpha_{-j}(B_w^{-i+j-1}x).
$$
On the other hand, by defining $\varphi$ as above, we have that
$\varphi \in C_{b,0}(X)$ and (\ref{ssnh2}) holds, which proves our claim.
Indeed, let us verify that $\varphi \in C_{b,0}(X)$ in the case
$X = \ell_p(\Z)$ (the case $X = c_0(\Z)$ will be left to the reader).
Note that $\varphi = f + g$, where
$$
f(x) = \sum_{i \in \N} \varphi_{-i}(x) e_{-i} \ \ \text{ and } \ \
g(x) = \sum_{i \in \N} \varphi_i(x) e_i.
$$
Hence, it is enough to prove that $f \in C_b(X)$ and $g \in C_b(X)$.
By (\ref{ssnh1}) and Lemma~\ref{Unilateral1},
there exist $s \in (0,1)$ and $\beta > 1$ such that
\begin{equation}\label{ssnh2b}
|w_{-j} w_{-j-1} \cdots w_{-j-k+1}| \leq \beta s^k \ \text{ and } \
\frac{1}{|w_j w_{j+1} \cdots w_{j+k-1}|} \leq \beta s^k
\ \ \text{ for all } j,k \in \N.
\end{equation}
Let $g_n(x) = \sum_{i=1}^n \varphi_i(x) e_i$ ($n \in \N$).
If we rearrange the double series that defines $g_n(x)$ by putting together
all terms that have the same number $t$ of $w_j$'s in the denominator
(a sort of ``diagonal ordering''), then we obtain
$$
g_n(x) = - \sum_{t=1}^n \Bigg(\sum_{k=t}^n
  \frac{\alpha_{k-t}(B_w^{t-1}x)}{w_{k-t+1} \cdots w_k} \,e_k\Bigg).
$$
Hence, it follows from Minkowski's inequality that
\begin{align}\label{ssnh2c}
\big(|\varphi_1(x)|^p + \cdots + |\varphi_n(x)|^p\big)^\frac{1}{p}
  &= \|g_n(x)\| \leq \sum_{t=1}^n \Bigg\|\sum_{k=t}^n
     \frac{\alpha_{k-t}(B_w^{t-1}x)}{w_{k-t+1} \cdots w_k} \,e_k\Bigg\|\\
  &= \sum_{t=1}^n \Bigg(\sum_{k=t}^n
     \Bigg|\frac{\alpha_{k-t}(B_w^{t-1}x)}{w_{k-t+1} \cdots w_k}\Bigg|^p
     \Bigg)^\frac{1}{p}\notag\\
  &\leq \sum_{t=1}^n \beta \|\alpha(B_w^{t-1}x)\| s^t
   \leq \frac{\beta}{1-s}\,\|\alpha\|_\infty < \infty,\notag
\end{align}
for all $x \in X$ and $n \in \N$. This shows that
$g\big(\ell_p(\Z)\big) \subset \ell_p(\Z)$ and that $g$ is bounded.
Since
$$
g(x) = -\sum_{t=1}^\infty \Bigg(\sum_{k=t}^\infty
     \frac{\alpha_{k-t}(B_w^{t-1}x)}{w_{k-t+1} \cdots w_k}\, e_k\Bigg),
$$
we have that
\begin{align*}
\|g(x)-g(y)\| &\leq \sum_{t=1}^\infty \Bigg(\sum_{k=t}^\infty
     \Bigg|\frac{\alpha_{k-t}(B_w^{t-1}x)-\alpha_{k-t}(B_w^{t-1}y)}
     {w_{k-t+1} \cdots w_k}\Bigg|^p\Bigg)^\frac{1}{p}\\
  &\leq \beta \sum_{t=1}^\infty \|\alpha(B_w^{t-1}x) - \alpha(B_w^{t-1}y)\|s^t,
\end{align*}
which easily implies the continuity of $g$. Thus, $g \in C_b(X)$.
The proof that $f \in C_b(X)$ is analogous, but we have to use the first
estimate in (\ref{ssnh2b}) instead of the second one.

Now, choose $0 < \eps < \|F^{-1}\|^{-1}$, let $\alpha \in C_b(X)$ be a
Lipschitz map with $\|\alpha\|_\infty \leq \eps$ and $\Lip(\alpha) \leq \eps$,
and put $S = B_w + \alpha$. We claim that there exists a unique
$u \in C_{b,0}(X)$ such that the map $h = I + u$ satisfies
\begin{equation}\label{ssnh3}
h \circ B_w = S \circ h.
\end{equation}
Indeed, let $G : C_{b,0}(X) \to C_{b,0}(X)$ be the map given by
$$
G(\varphi) = F^{-1}(\alpha \circ (I + \varphi)).
$$
Since
$$
\|G(\psi) - G(\varphi)\|_\infty \leq \Lip(\alpha) \|F^{-1}\|
\|\psi-\varphi\|_\infty,
$$
we have that $G$ is a contraction, and so it has a unique fixed point in
$C_{b,0}(X)$. Since (\ref{ssnh3}) is equivalent to $G(u) = u$, our second
claim is proved.

Now, we claim that there is a unique $v \in C_{b,0}(X)$ such that the
map $h' = I + v$ satisfies
\begin{equation}\label{ssnh4}
h' \circ S = B_w \circ h'.
\end{equation}
Indeed, (\ref{ssnh4}) is equivalent to the equation
$$
v \circ S - B_w \circ v = - \alpha.
$$
With $v_0 = 0$, this equation implies that
$$
v_i(x) = \sum_{j=0}^{i-1} \frac{\alpha_j(S^{i-j-1}x)}{w_{j+1} \cdots w_i}
\ \text{ and } \
v_{-i}(x) = -\sum_{j=1}^i w_{-i+1} \cdots w_{-j} \, \alpha_{-j}(S^{-i+j-1}x),
$$
for every $i \in \N$. So, if there is a solution $v$ in $C_{b,0}(X)$,
then it must be unique. However, by defining $v$ as above, we have that
$v \in C_{b,0}(X)$ and $h' = I + v$ satisfies (\ref{ssnh4}).
This proves our third claim.

By (\ref{ssnh3}) and (\ref{ssnh4}),
\begin{equation}\label{ssnh5}
h' \circ h \circ B_w = B_w \circ h' \circ h.
\end{equation}
Since $h' \circ h = I + \psi$ with $\psi \in C_{b,0}(X)$, (\ref{ssnh5})
is equivalent to
$$
\psi_i(B_wx) = w_{i+1} \psi_{i+1}(x) \ \ \text{ for all } i \in \Z.
$$
Since $\psi_0 = 0$, this clearly implies that $\psi = 0$. Thus,
\begin{equation}\label{ssnh6}
h' \circ h = I.
\end{equation}

Let us now prove that
\begin{equation}\label{ssnh7}
h \circ h' = I.
\end{equation}
It is enough to show that the only solution to
\begin{equation}\label{ssnh8}
(I + \theta) \circ S = S \circ (I + \theta),
\end{equation}
with $\theta \in C_{b,0}(X)$, is $\theta = 0$.
For this purpose, consider the bounded linear map
$$
H : \gamma \in C_{b,0}(X) \mapsto
    \gamma \circ S - B_w \circ \gamma \in C_b(X).
$$
We have that $H$ is bijective. Indeed, given $\eta \in C_b(X)$, the equation
$$
\gamma(Sx) - B_w(\gamma(x)) = \eta(x)
$$
is equivalent to
$$
\gamma_i(Sx) - w_{i+1} \gamma_{i+1}(x) = \eta_i(x) \ \
\text{ for all } i \in \Z,
$$
and so we can argue as in the case of the map $F$.
Define $K : C_{b,0}(X) \to C_{b,0}(X)$ by
$$
K(\varphi) = H^{-1}(\alpha \circ (I + \varphi) - \alpha).
$$
We have that
$$
\|K(\psi) - K(\varphi)\|_\infty \leq \Lip(\alpha) \|H^{-1}\|
\|\psi-\varphi\|_\infty.
$$
Moreover, we know that
$$
H^{-1}(\eta)(x) = \sum_{i \in \N} \gamma_{-i}(x) e_{-i}
                + \sum_{i \in \N} \gamma_i(x) e_i,
$$
where
$$
\gamma_i(x) = -\sum_{j=0}^{i-1} \frac{\eta_j(S^{i-j-1}x)}{w_{j+1} \cdots w_i}
\ \text{ and } \
\gamma_{-i}(x) = \sum_{j=1}^i w_{-i+1} \cdots w_{-j} \, \eta_{-j}(S^{-i+j-1}x).
$$
A computation analogous to (\ref{ssnh2c}) shows that
$$
\|H^{-1}(\eta)\|_\infty \leq \frac{2\beta}{1-s}\,\|\eta\|_\infty.
$$
This shows that, although $H$ depends on $\eps$, $\|H^{-1}\|$ is bounded
by a constant which is independent of $\eps$. Thus, by choosing $\eps > 0$
small enough, we can guarantee that $K$ is a contraction, and so it has
a unique fixed point in $C_{b,0}(X)$. Since
$K(0) = 0$ and (\ref{ssnh8}) is equivalent to $K(\theta) = \theta$,
we conclude that $\theta = 0$ is the only solution to (\ref{ssnh8})
in $C_{b,0}(X)$.

By (\ref{ssnh6}) and (\ref{ssnh7}), $h$ is a homeomorphism, which proves
that $B_w$ is structurally stable.
We mentioned in the introduction and will prove in the next section that,
in the present case, $B_w$ satisfies the frequent hypercyclicity criterion.
Hence, it cannot be structurally stable relative to $GL(X)$ in view of
Lemma~\ref{Hypercyclic}.
\hfill $\Box$

\begin{remark}
An invertible operator $T$ on a Banach space $X$ is said to be
{\em strongly structurally stable} \cite{JRob72} if for every $\gamma > 0$
there exists $\eps > 0$ such that the following property holds:
for any Lipschitz map $\varphi \in C_b(X)$ with
$\|\varphi\|_\infty \leq \eps$ and $\Lip(\varphi) \leq \eps$,
there is a homeomorphism $h : X \to X$ such that
$h \circ T = (T + \varphi) \circ h$ and $\|h - I\|_\infty \leq \gamma$.
So, it is now required that the homeomorphism $h$ conjugating $T$ and
$T + \varphi$ is close to the identity operator.
The proof of the fact that invertible hyperbolic operators are structurally
stable given in \cite{CPug69}, for instance, shows that these operators
are in fact strongly structurally stable. We also observe that the proof
of Theorem~\ref{SSShifts} actually shows that $B_w$ is strongly structurally
stable whenever $w$ satisfies condition~(C) of Theorem~\ref{Characterization3}.
Indeed, recall that the homeomorphism $h$ conjugating $B_w$ and $S$
constructed in the proof of Theorem~\ref{SSShifts} has the form
$h = I + u$, where $u$ is the unique fixed point of the map $G$. Hence,
$$
\|u\|_\infty = \|G(u)\|_\infty = \|F^{-1}(\alpha \circ (I + \varphi))\|_\infty
      \leq \|F^{-1}\| \|\alpha\|_\infty \leq \|F^{-1}\| \eps.
$$
Hence, we can make $\|u\|_\infty$ as small as we want by choosing $\eps$
small enough.
\end{remark}

\begin{remark}
{\bf (a)}
Note that the homeomorphism $h$ of the form $I + u$ constructed in the proof
of Theorem~\ref{SSShifts} is not unique, contrary to what happens in the
hyperbolic case. Indeed, if instead of the subspace $C_{b,0}(X)$ we consider
the subspace $C_{b,j}(X) = \{\varphi \in C_b(X) : \varphi_j = 0\}$ for some
$j \neq 0$, then we would obtain a different $u$ and so a different $h$.\\
{\bf (b)}
The proof of Theorem~\ref{SSShifts} works as well if we replace the space
$C_b(X)$ by the space $U_b(X)$ of all bounded uniformly continuous maps
$\varphi : X \to X$ endowed with the supremum norm. Thus, the homeomorphism
$h = I + u$ conjugating $B_w$ and $S$ can be chosen to be a uniform
homeomorphism.
\end{remark}

We saw in Theorem~\ref{SSShifts} that for the class of invertible weighted
shifts on the spaces $\ell_p(\Z)$ ($1 \leq p < \infty$) and $c_0(\Z)$,
shadowing implies structural stability. This is also true for invertible
operators on finite-dimensional spaces. These remarks motivate the
following question.

\begin{question}
Is it true that every invertible operator on a Banach space that has the
shadowing property is structurally stable? How about the converse?
\end{question}

\begin{proposition}
Let $X = \ell_p(\Z)$ $(1 \leq p < \infty)$ or $X = c_0(\Z)$.
If an invertible bilateral weighted backward shift $B_w$ on $X$
is expansive and strongly structurally stable, then it is hyperbolic.
\end{proposition}

\begin{proof}
We assume that $B_w$ is expansive, strongly structurally stable
and not hyperbolic, and we will reach a contradiction.
By Theorem~\ref{HyperSS}(a), $B_w$ is uniformly expansive.
It was proved in \cite[Theorem~E]{BerCirDarMesPuj18} that $B_w$ is
uniformly expansive if and only if one of the following conditions holds:
\begin{itemize}
\item [(a)] $\displaystyle \lim_{n \to \infty} \inf_{k \in \Z}
             |w_{k+1} \cdots w_{k+n}|^{-1} = \infty$.
\item [(b)] $\displaystyle \lim_{n \to \infty} \inf_{k \in \Z}
             |w_{k-n+1} \cdots w_k| = \infty$.
\item [(c)] $\displaystyle \lim_{n \to \infty} \inf_{k \in \N}
             |w_{k+1} \cdots w_{k+n}|^{-1} = \infty$ and
            $\displaystyle \lim_{n \to \infty} \inf_{k \in -\N}
             |w_{k-n+1} \cdots w_k| = \infty$.
\end{itemize}
By the version of Lemma~\ref{Unilateral1} with $\Z$ instead of $\N$,
we see that (a) and (b) are equivalent to
$$
\lim_{n \to \infty} \sup_{k \in \Z} |w_{k+1} \cdots w_{k+n}|^\frac{1}{n} < 1
\ \ \text{ and } \ \
\lim_{n \to \infty} \sup_{k \in \Z} |w_{k-n+1} \cdots w_k|^{-\frac{1}{n}} < 1,
$$
respectively, which are exactly the cases in which $B_w$ is hyperbolic
\cite[Remark~35]{BerCirDarMesPuj18}. Since we are assuming that $B_w$
is not hyperbolic, we conclude that (c) must hold. In particular,
\begin{equation}\label{CCC}
\lim_{i \to \infty} \frac{1}{|w_1 \cdots w_i|} = \infty
\ \ \text{ and } \ \
\lim_{i \to \infty} |w_{-i+1} \cdots w_0| = \infty.
\end{equation}
Let $\eps > 0$ be associated to $\gamma = 1$ according to the definition
of strong structural stability. Let $\alpha \in C_b(X)$ be the constant
map given by $\alpha(x) = (\eps/2)e_0$. There exists $\varphi \in C_b(X)$
with $\|\varphi\|_\infty \leq 1$ such that $h = I + \varphi$ is a
homeomorphism satisfying
$$
h \circ B_w = (B_w + \alpha) \circ h.
$$
This equation is equivalent to
$$
\varphi(B_wx) - B_w(\varphi(x)) = \alpha(x + \varphi(x)),
$$
which is the same as (\ref{ssnh2}) since $\alpha$ is constant.
Thus, $\varphi$ must satisfy equations (\ref{AAA}) and (\ref{BBB}), and so
$$
\varphi_i(x) = \frac{\varphi_0(B_w^ix)}{w_1 \cdots w_i}
              -\frac{\eps}{2 w_1 \cdots w_i}
\ \ \text{ and } \ \
\varphi_{-i}(x) = w_{-i+1} \cdots w_0\, \varphi_0(B_w^{-i}x),
$$
for each $i \in \N$. Since $\varphi$ is bounded, the second equality
in (\ref{CCC}) shows that
$$
\varphi_0(x) = \frac{\varphi_{-i}(B_w^ix)}{w_{-i+1} \cdots w_0}
             \to 0 \ \text{ as } i \to \infty,
$$
proving that $\varphi_0$ is identically zero. Hence,
$$
\varphi_i(x) = -\frac{\eps}{2 w_1 \cdots w_i} \ \ \text{ for all } i \in \N.
$$
Now, the first equality in (\ref{CCC}) leads to a contradiction with the
boundedness of $\varphi$. This completes the proof.
\end{proof}

The previous results motivate the following question.

\begin{question}
If an invertible operator on a Banach space is expansive and structurally
stable, is it necessarily hyperbolic?
\end{question}

The following question also seems to be open.

\begin{question}
Do hyperbolicity and structural stability relative to $GL(X)$ coincide?
\end{question}

%%%%%%%%%%%%%%%%%%%%%%%%%%%%%%%%%%%%%%%%%%%%%%%%%%%%%%%%%%%%%%%%%%%%%%%%%%%%%

\section{Weighted shifts with the shadowing property}\label{Shifts}

\begin{theorem}\label{Characterization3}
Let $X = \ell_p(\Z)$ $(1 \leq p < \infty)$ or $X = c_0(\Z)$.
Let $w = (w_n)_{n \in \Z}$ be a bounded sequence of scalars with
$\inf_{n \in \Z} |w_n| > 0$ and consider the bilateral weighted backward
shift
$$
B_w : (x_n)_{n \in \Z} \in X \mapsto (w_{n+1}x_{n+1})_{n \in \Z} \in X.
$$
Then $B_w$ has the shadowing property if and only if one of the
following conditions holds:
\begin{itemize}
\item [\rm (A)] $\displaystyle \lim_{n \to \infty} \sup_{k \in \Z}
                |w_k w_{k+1} \cdot\ldots\cdot w_{k+n}|^\frac{1}{n} < 1$.
\item [\rm (B)] $\displaystyle \lim_{n \to \infty} \inf_{k \in \Z}
                |w_k w_{k+1} \cdot\ldots\cdot w_{k+n}|^\frac{1}{n} > 1$.
\item [\rm (C)] $\displaystyle \lim_{n \to \infty} \sup_{k \in \N}
                |w_{-k} w_{-k-1} \cdot \ldots \cdot w_{-k-n}|^\frac{1}{n} < 1$
                and
                $\displaystyle \lim_{n \to \infty} \inf_{k \in \N}
                |w_k w_{k+1} \cdot \ldots \cdot w_{k+n}|^\frac{1}{n} > 1$.
\end{itemize}
\end{theorem}

\smallskip
We recall that (A) and (B) are exactly the cases in which the operator
$B_w$ is hyperbolic \cite[Remark~35]{BerCirDarMesPuj18}. In these cases,
$B_w$ does not exhibit chaotic behaviour, since hyperbolic operators are
not Li-Yorke chaotic \cite[Theorem~C]{BerCirDarMesPuj18}. On the other hand,
we will prove after the proof of Theorem~\ref{Characterization3} that
condition (C) implies the frequent hypercyclicity criterion. Hence,
in this case, $B_w$ exhibit several types of chaotic behaviours,
including mixing, Devaney chaos and dense distributional chaos.

\begin{lemma}\label{Unilateral1}
If $(w_n)_{n \in \N}$ is a bounded sequence of scalars, then the following
assertions are equivalent:
\begin{itemize}
\item [\rm   (i)] $\displaystyle \lim_{n \to \infty} \sup_{k \in \N}
      |w_k w_{k+1} \cdot\ldots\cdot w_{k+n}|^\frac{1}{n} < 1$;
\item [\rm  (ii)] $\displaystyle \sup_{k \in \N}
      |w_k w_{k+1} \cdot\ldots\cdot w_{k+n-1}| < 1$ for some $n \in \N$;
\item [\rm (iii)] $\displaystyle \sup_{k \in \N} \sum_{n=0}^\infty
      |w_k w_{k+1} \cdot\ldots\cdot w_{k+n}| < \infty$;
\item [\rm  (iv)] $\displaystyle \sup_{k \in \N} \sum_{n=0}^{k-1}
      |w_k w_{k-1} \cdot\ldots\cdot w_{k-n}| < \infty$.
\end{itemize}
\end{lemma}

\begin{proof}
It is clear that (i) implies (ii).
The fact that (ii) implies (iii) and (iv) follows easily by comparing
the series in (iii) and (iv) with a suitable geometric series.
Let us show that (iii) implies (i). For each $k \in \N$, let
$R_k = \sum_{n=0}^\infty |w_k w_{k+1} \cdot\ldots\cdot w_{k+n}|$.
By hypothesis, there is a constant $K > 0$ such that $R_k \leq K$
for all $k \in \N$. Since $|w_k| = \frac{R_k}{1 + R_{k+1}}$ ($k \in \N$),
$$
|w_k w_{k+1} \cdot\ldots\cdot w_{k+n}|
  = \frac{R_k}{1 + R_{k+n+1}} \cdot \frac{R_{k+1}}{1 + R_{k+1}}
   \cdot\ldots\cdot \frac{R_{k+n}}{1 + R_{k+n}}
  \leq K \cdot \Big(\frac{K}{1+K}\Big)^n,
$$
where we used the fact that $\frac{R_j}{1+R_j} = 1 - \frac{1}{1+R_j}
\leq \frac{K}{1+K}$ for all $j \in \N$. This implies (i).
The proof that (iv) implies (i) is similar and is left to the reader.
\end{proof}

\smallskip
\noindent
{\it Proof of Theorem \ref{Characterization3}.}
Since invertible hyperbolic operators have the shadowing property,
in order to prove the sufficiency it is enough to show that (C) implies
shadowing. Suppose that (C) holds. Let
$$
M = \big\{(x_n)_{n \in \Z} \in X : x_n = 0 \text{ for all } n > 0\big\}
$$
and
$$
N = \big\{(x_n)_{n \in \Z} \in X : x_n = 0 \text{ for all } n \leq 0\},
$$
which are closed subspaces of $X$ such that $X = M \oplus N$,
$B_w(M) \subset M$ and $B_w^{-1}(N) \subset N$. Since, for each $n \in \N$,
$$
\|(B_w|_M)^n\| =
  \sup_{k \in \N_0} |w_{-k} w_{-k-1} \cdot\ldots\cdot w_{-k-n+1}|
$$
and
$$
\|(B_w^{-1}|_N)^n\| =
  \sup_{k \in \N} |w_{k+1} w_{k+2} \cdot\ldots\cdot w_{k+n}|,
$$
we have that condition (C) is equivalent to say that
$$
\sigma(B_w|_M) \subset \D \ \ \ \text{ and } \ \ \
\sigma(B_w^{-1}|_N) \subset \D,
$$
where $\D$ denotes the open unit disk in the complex plane. Thus, by
\cite[Theorem~A]{BerCirDarMesPuj18}, $B_w$ has the shadowing property.

Conversely, assume that $B_w$ has the shadowing property. We claim that
one of the following conditions must be true:
\begin{itemize}
\item [\rm (C1)] $\displaystyle \lim_{n \to \infty} \sup_{k \in \N}
      |w_k w_{k+1} \cdot\ldots\cdot w_{k+n}|^\frac{1}{n} < 1$.
\item [\rm (C2)] $\displaystyle \lim_{n \to \infty} \inf_{k \in \N}
      |w_k w_{k+1} \cdot\ldots\cdot w_{k+n}|^\frac{1}{n} > 1$.
\end{itemize}
Indeed, assume that condition (C1) is false. We will prove that condition
(C2) must be true. Take $\eps = 1$ and let $\delta > 0$ be associated to
this $\eps$ according to the definition of shadowing. Since (C1) is false,
Lemma~\ref{Unilateral1} guarantees that there exist $t \in \N$ and
$m_0 \in \N$ such that
\begin{equation}
\sum_{n=0}^{m_0} |w_t w_{t+1} \cdot\ldots\cdot w_{t+n}|
  \geq \frac{1+\delta}{\delta^2}\cdot \label{C1}
\end{equation}
Given $m > m_0$, we construct a $\delta$-pseudotrajectory of $B_w$ in the
following way:
$$x_0 = e^{i\theta_0} e_{t+m},\;
 x_k= B_w(x_{k-1}) + \delta e^{i\theta_k} e_{t+m-k},\; k=1,\ldots, m$$
and
$$ 
x_n = B_w(x_{n-1}) \ \text{ for all } n \geq m+1,$$
where the angles $\theta_0,\theta_1,\ldots,\theta_m$ are chosen so that
$$
e^{i\theta_k} w_t w_{t+1} \cdot\ldots\cdot w_{t+m-k}
  = |w_t w_{t+1} \cdot\ldots\cdot w_{t+m-k}| \ \
  \text{ for each } 0 \leq k \leq m.
$$
By shadowing, there exists $a = (a_n)_{n \in \Z} \in X$ such that
\begin{equation}
\|x_n - B_w^n(a)\| < 1 \ \ \text{ for all } n \in \Z. \label{C2}
\end{equation}
Since $B_w e_k = w_k e_{k-1}$ for all $k \in \Z$,
the $(t-1)^\text{th}$ coordinate of $x_{m+1}$ is
$$
e^{i\theta_0} w_t w_{t+1} \cdot\ldots\cdot w_{t+m}
+ \big(e^{i\theta_1} w_t w_{t+1} \cdot\ldots\cdot w_{t+m-1} + \cdots +
       e^{i\theta_{m-1}} w_t w_{t+1} + e^{i\theta_m} w_t\big) \delta,
$$
which is equal to
$$
|w_t w_{t+1} \cdot\ldots\cdot w_{t+m}|
+ \big(|w_t w_{t+1} \cdot\ldots\cdot w_{t+m-1}|
+ \cdots + |w_t w_{t+1}| + |w_t|\big) \delta.
$$
By writing $a_{t+m} = e^{i\theta_0} + \gamma$, we have that $|\gamma| < 1$
and the $(t-1)^\text{th}$ coordinate of $B_w^{m+1}(a)$ is
$$
w_t w_{t+1} \cdot\ldots\cdot w_{t+m} \cdot (e^{i\theta_0} + \gamma).
$$
Hence, (\ref{C2}) gives
\begin{equation}
\big|\big(|w_t w_{t+1} \cdot\ldots\cdot w_{t+m-1}| + \cdots + |w_t w_{t+1}|
 + |w_t|\big) \delta - w_t w_{t+1} \cdot\ldots\cdot w_{t+m}\gamma\big| < 1.
 \label{C3}
\end{equation}
By (\ref{C1}) and (\ref{C3}),
$$
|w_t w_{t+1} \cdot\ldots\cdot w_{t+m}| > 1/\delta.
$$
Again by (\ref{C3}),
$$
\big(|w_t w_{t+1} \cdot\ldots\cdot w_{t+m-1}| + \cdots + |w_t w_{t+1}|
 + |w_t|\big) \delta < 1 + |w_t w_{t+1} \cdot\ldots\cdot w_{t+m}| |\gamma|.
$$
By dividing both sides by $|w_t w_{t+1} \cdot\ldots\cdot w_{t+m}|\delta$,
we obtain
$$
\frac{1}{|w_{t+m}|} + \frac{1}{|w_{t+m}w_{t+m-1}|} + \cdots +
\frac{1}{|w_{t+m}w_{t+m-1} \cdot\ldots\cdot w_{t+1}|} < 1 + \frac{1}{\delta}
\cdot
$$
Since this is true for every $m > m_0$, it follows that
$\inf_{n \in \N} |w_n| > 0$ and, by Lemma~\ref{Unilateral1},
condition (C2) holds.

Now, the inverse of $B_w$ is the bilateral weighted forward shift
$$
F_{w'} : (x_n)_{n \in \Z} \in X \mapsto (w'_{n-1}x_{n-1})_{n \in \Z} \in X,
$$
where $w'_n = 1/w_{n+1}$ for each $n \in \Z$. On the other hand, the
isometric isomorphism
$\phi : (x_n)_{n \in \Z} \in X \mapsto (x_{-n})_{n \in \Z} \in X$
establishes a conjugation between $F_{w'}$ and $B_{w''}$ (that is,
$\phi \circ F_{w'} = B_{w''} \circ \phi$), where
$$
w''_n = w'_{-n} = \frac{1}{w_{-n+1}} \ \ \text{ for each } n \in \Z.
$$
Hence, $B_{w''}$ also has the shadowing property and, by the above
reasoning, we conclude that one of the following properties must be true:
\begin{itemize}
\item [\rm (D1)] $\displaystyle \lim_{n \to \infty} \inf_{k \in \N}
      |w_{-k} w_{-k-1} \cdot\ldots\cdot w_{-k-n}|^\frac{1}{n} > 1$.
\item [\rm (D2)] $\displaystyle \lim_{n \to \infty} \sup_{k \in \N}
      |w_{-k} w_{-k-1} \cdot\ldots\cdot w_{-k-n}|^\frac{1}{n} < 1$.
\end{itemize}
Thus, we have four possibilities. Note that condition (A) is equivalent
to say that ``(C1) and (D2) are true'', condition (B) is equivalent to
say that ``(C2) and (D1) are true'', and condition (C) means that
``(C2) and (D2) are true''. Consequently, it remains to show that
(C1) and (D1) both true is not possible. Indeed, assume that both
(C1) and (D1) hold. Let $\delta$ and $\eps$ be as above. We construct a
$\delta$-pseudotrajectory of $B_w$ as follows:
$$
y_0 = e_0, \ \ y_n = B_w(y_{n-1}) + \delta e_0 \text{ for } n \geq 1, \ \
y_n = B_w^{-1}(y_{n+1} + \delta e_0) \text{ for } n \leq -1.
$$
Note that
$$
y_n = w_0 \cdots w_{-n+1} e_{-n} + \delta w_0 \cdots w_{-n+2} e_{-n+1}
      + \delta w_0 \cdots w_{-n+3} e_{-n+2} + \cdots + \delta w_0 e_{-1}
      + \delta e_0
$$
and
$$
y_{-n} = \frac{\delta + 1}{w_1 \cdots w_n}\, e_n
       + \frac{\delta}{w_1 \cdots w_{n-1}}\, e_{n-1}
       + \frac{\delta}{w_1 \cdots w_{n-2}}\, e_{n-2}
       + \cdots + \frac{\delta}{w_1}\, e_1,
$$
for each $n \in \N$. There exists $b = (b_n)_{n \in \Z} \in X$ such that
\begin{equation}
\|y_n - B_w^n(b)\| < 1 \ \ \text{ for every } n \in \Z. \label{C4}
\end{equation}
For each $k \geq 1$ and each $n \geq 0$, since the $(-k-(n+1))^\text{th}$
coordinate of $B_w^{n+1}(b)$ is equal to
$w_{-k-n} \cdot\ldots\cdot w_{-k-1}w_{-k} b_{-k}$,
(\ref{C4}) implies that
$$
|w_{-k-n} \cdot\ldots\cdot w_{-k-1}w_{-k} b_{-k}| < 1.
$$
Hence, (D1) gives $b_{-k} = 0$ for all $k \geq 1$. Analogously, (C1) gives
$b_k = 0$ for all $k \geq 1$. Thus, $b = b_0 e_0$, which contradicts
(\ref{C4}). This completes the proof of Theorem~\ref{Characterization3}.
\hfill $\Box$

\vspace{1em}

Let us now show that condition (C) implies the frequent hypercyclicity
criterion. Recall that an operator $T$ on a separable Banach space
$Y$ is said to satisfy the {\em frequent hypercyclicity criterion} if
there exist a dense subset $Y_0$ of $Y$ and a map $S : Y_0 \to Y_0$ such that
the following properties hold for every $y \in Y_0$:
\begin{itemize}
\item [\rm   (i)] $\sum_{k=0}^\infty T^ky$ converges unconditionally,
\item [\rm  (ii)] $\sum_{k=0}^\infty S^ky$ converges unconditionally,
\item [\rm (iii)] $TSy = y$.
\end{itemize}
In our case, $Y = X$ and $T = B_w$. We define $Y_0$ as being the set of all
sequences $y = (y_n)_{n \in \Z} \in Y$ that have finite support and $S$ as
being the inverse $B_w^{-1}$ of $B_w$, which is equal to $F_{w'}$, where
$w' = (1/w_{n+1})_{n \in \Z}$. Clearly $Y_0$ is dense in $Y$ and (iii) holds.
In order to prove (i), fix $y = (y_n)_{n \in \Z} \in Y_0$ and let $r \in \N$
be such that $y_n = 0$ whenever $|n| > r$. Since
$$
T^ky = \big(w_{n+1} \cdot\ldots\cdot w_{n+k} y_{n+k}\big)_{n \in \Z},
$$
we have that
\begin{align*}
\sum_{k=1}^\infty \|T^ky\|
&\leq \sum_{k=1}^\infty \sum_{n = -r-k}^{r-k}
  |w_{n+1} \cdot\ldots\cdot w_{n+k}| |y_{n+k}|\\
&= \sum_{m = -r}^r \Bigg(\sum_{k=1}^\infty
  |w_{m-k+1} \cdot\ldots\cdot w_m|\Bigg) |y_m| < \infty,
\end{align*}
because of Lemma~\ref{Unilateral1}. In a similar way,
$\displaystyle \sum_{k=1}^\infty \|S^ky\| < \infty$.

\medskip
We have the following sufficient condition for positive shadowing, which
was inspired by \cite[Theorem~A]{BerCirDarMesPuj18}.

\begin{proposition}\label{PS}
Let $T$ be an operator on a Banach space $X$. Suppose that $X = M \oplus N$,
where $M$ and $N$ are closed subspaces of $X$ with the following properties:
\begin{itemize}
\item $T(M) \subset M$ and $\sigma(T|_M) \subset \D$;
\item $T|_N : N \to T(N)$ is bijective, $T(N)$ is closed and
      contains $N$, and $\sigma((T|_N)^{-1}|_N) \subset \D$.
\end{itemize}
Then $T$ has the positive shadowing property.
\end{proposition}

\begin{proof}
Let $S$ be the operator $(T|_N)^{-1}|_N$ on $N$.
For each $x \in X$, we can write $x = x^{(1)} + x^{(2)}$ with unique
$x^{(1)} \in M$ and $x^{(2)} \in N$. Moreover, there exists $\beta > 0$
such that $\|x^{(1)}\| \leq \beta \|x\|$ and $\|x^{(2)}\| \leq \beta \|x\|$
for all $x \in X$. Since $r(T|_M) < 1$ and $r(S) < 1$, there exist
$0 < t < 1$ and $C \geq 1$ such that $\|T^ny\| \leq C\, t^n \|y\|$ and
$\|S^nz\| \leq C\, t^n \|z\|$ whenever $n \in \N_0$, $y \in M$ and $z \in N$.
By linearity, it is enough to consider $\eps = 1$ in the definition of
positive shadowing. Put $\delta = \frac{1-t}{4C\beta}$ and let
$(x_n)_{n \in \N_0}$ be a $\delta$-pseudotrajectory of $T$.
For each $n \in \N_0$, let $z_n = x_{n+1} - Tx_n$ and let
$$
y^{(1)}_n = \sum_{k=0}^{n-1} T^kz^{(1)}_{n-k-1} \in M, \ \
y^{(2)}_n = - \sum_{k=1}^\infty S^kz^{(2)}_{n+k-1} \in N
\ \text{ and } \ y_n = y^{(1)}_n + y^{(2)}_n \in X,
$$
where $y^{(1)}_0 = 0$. Then,
$$
Ty_0 + z_0 = - \sum_{k=1}^\infty S^{k-1}z^{(2)}_{k-1} + z^{(1)}_0 + z^{(2)}_0
           = z^{(1)}_0 - \sum_{k=1}^\infty S^kz^{(2)}_k = y_1
$$
and, for every $n \geq 1$,
\begin{align*}
Ty_n + z_n &= \sum_{k=0}^{n-1} T^{k+1}z^{(1)}_{n-k-1}
            - \sum_{k=1}^\infty S^{k-1}z^{(2)}_{n+k-1}
            + z^{(1)}_n + z^{(2)}_n\\
           &= \sum_{k=0}^n T^kz^{(1)}_{n-k}
            - \sum_{k=1}^\infty S^kz^{(2)}_{n+k} = y_{n+1}.
\end{align*}
Therefore, $x_{n+1} - y_{n+1} = T(x_n - y_n)$, and so
$x_n - y_n = T^n(x_0 - y_0)$ for all $n \in \N_0$. Since
$$
\|y^{(1)}_n\| \leq C \sum_{k=0}^{n-1} t^k \|z^{(1)}_{n-k-1}\|
\ \ \text{ and } \ \
\|y^{(2)}_n\| \leq C \sum_{k=1}^\infty t^k \|z^{(2)}_{n+k-1}\|,
$$
we conclude that
$$
\|x_n - T^n(x_0-y_0)\| = \|y_n\| \leq \frac{2 C \beta \delta}{1-t}
 = \frac{1}{2} < \eps \ \text{ for all } n \in \N_0.
$$
This proves that $T$ has the positive shadowing property.
\end{proof}

In the next two propositions we will characterize positive shadowing for
unilateral (backward and forward) weighted shifts.

\begin{proposition}\label{Characterization1}
Let $X = \ell_p(\N)$ $(1 \leq p < \infty)$ or $X = c_0(\N)$.
Let $w = (w_n)_{n \in \N}$ be a bounded sequence of nonzero scalars
and consider the unilateral weighted backward shift
$$
B_w : (x_1,x_2,\ldots) \in X \mapsto (w_2x_2,w_3x_3,\ldots) \in X.
$$
Then $B_w$ has the positive shadowing property if and only if one of the
following conditions holds:
\begin{itemize}
\item [\rm (a)] $\displaystyle \lim_{n \to \infty} \sup_{k \in \N}
      |w_k w_{k+1} \cdot\ldots\cdot w_{k+n}|^\frac{1}{n} < 1$.
\item [\rm (b)] $\displaystyle \lim_{n \to \infty} \inf_{k \in \N}
      |w_k w_{k+1} \cdot\ldots\cdot w_{k+n}|^\frac{1}{n} > 1$.
\end{itemize}
\end{proposition}

We observe that condition (a) means that $B_w$ is hyperbolic
\cite[Remark~35]{BerCirDarMesPuj18} and that condition (b) implies the
{\em frequent hypercyclicity criterion}.

\begin{proof}
If condition (b) holds, then we apply Proposition \ref{PS} with
$$
M = \{(x_1,0,0,\ldots) : x_1 \in \K\} \ \ \text{ and } \ \
N = \{(0,x_2,x_3,\ldots) : (x_2,x_3,\ldots) \in X\}.
$$
Clearly, $B_w(M) = \{0\} \subset M$ and $\sigma(B_w|_M) = \{0\} \subset \D$.
Condition (b) implies that $\inf_{n \in \N} |w_n| > 0$, and so
$T|_N : N \to T(N)$ is bijective. Let $S = (T|_N)^{-1}|_N$. Since
$$
S\big((0,x_2,x_3,\ldots)\big)
  = \Big(0,0,\frac{x_2}{w_3},\frac{x_3}{w_4},\ldots\Big),
$$
we have that
$$
r(S) = \lim_{n \to \infty} \sup_{k \geq 3}
       |w_k \cdot \ldots \cdot w_{k+n-1}|^{-\frac{1}{n}}.
$$
Thus, condition (b) is equivalent to say that $\sigma(S) \subset \D$.
On the other hand, condition (a) means that $B_w$ is hyperbolic
\cite[Remark~35]{BerCirDarMesPuj18} and every hyperbolic operator has the
positive shadowing property \cite[Theorem~13]{BerCirDarMesPuj18}.

For the converse, we assume that $B_w$ has the positive shadowing property
and argue as in the second paragraph of the proof of
Theorem~\ref{Characterization3} to conclude that either (a) or (b) holds.
\end{proof}

\begin{proposition}\label{Characterization2}
Let $X = \ell_p(\N)$ $(1 \leq p < \infty)$ or $X = c_0(\N)$.
Let $w = (w_n)_{n \in \N}$ be a bounded sequence of nonzero scalars
and consider the unilateral weighted forward shift
$$
F_w : (x_1,x_2,\ldots) \in X \mapsto (0,w_1x_1,w_2x_2,\ldots) \in X.
$$
Then $F_w$ has the positive shadowing property if and only if it is
hyperbolic, that is,
$$
\lim_{n \to \infty} \sup_{k \in \N}
  |w_k w_{k+1} \cdot\ldots\cdot w_{k+n}|^\frac{1}{n} < 1.
$$
\end{proposition}

\begin{proof}
Suppose that $F_w$ is not hyperbolic and fix $\delta > 0$.
By Lemma~\ref{Unilateral1}, there exists $t \in \N$ such that
$$
|w_t| + |w_t w_{t-1}| + \cdots +
|w_t w_{t-1} \cdot\ldots\cdot w_1| \geq 1/\delta.
$$
We construct a $\delta$-pseudotrajectory of $F_w$ in the following way:
$
x_0 = 0$
and
$$x_k= F_w(x_{k-1}) + \delta e^{i\theta_k} e_k,\; 1 \leq k \leq t,\;
x_n = F_w(x_{n-1}) \ \text{ for all } n \geq t+1,$$
where the angles $\theta_1,\ldots,\theta_t$ are chosen so that
the $(t+1)^\text{th}$ coordinate of $x_{t + 1}$ is equal to
$$
\big(|w_1 \cdot\ldots\cdot w_t| + |w_2 \cdot\ldots\cdot w_t|
  + \cdots + |w_{t-1} w_t| + |w_t|\big) \delta.
$$
Since this number is $\geq 1$, $(x_n)_{n \in \N_0}$ cannot be $1$-shadowed
by a real trajectory of $F_w$.
\end{proof}

%%%%%%%%%%%%%%%%%%%%%%%%%%%%%%%%%%%%%%%%%%%%%%%%%%%%%%%%%%%%%%%%%%%%%%%%%%%%%

\section*{Appendix}

\begin{theorem}
Every invertible hyperbolic operator $T$ on a Banach space $X$ is
structurally stable relative to $GL(X)$.
\end{theorem}

\begin{proof}
Since $T$ is structurally stable, there exists $\eps > 0$ such that
$T + \varphi$ is topologically conjugate to $T$ whenever $\varphi \in C_b(X)$
is a Lipschitz map with $\|\varphi\|_\infty \leq \eps$ and
$\Lip(\varphi) \leq \eps$. Take $S \in GL(X)$ with $\|S - T\| \leq \eps/3$.
By choosing $\eps$ small enough, we have that $S$ is also hyperbolic.
We will find a homeomorphism $h : X \to X$ such that
\begin{equation}\label{SS1}
h \circ T = S \circ h.
\end{equation}
Let $U$ denote the open unit ball in $X$. By Lemma~\ref{Extension},
there is a Lipschitz map $\varphi \in C_b(X)$ such that $\varphi = S - T$
on $U$, $\|\varphi\|_\infty \leq \eps$ and $\Lip(\varphi) \leq \eps$.
Hence, by our choice of $\eps$, there is a homeomorphism $f : X \to X$
such that $f \circ T = (T + \varphi) \circ f$. Since $\varphi(0) = 0$,
we must have $f(0) = 0$. Consider the open set $V = f^{-1}(U)$. Then
\begin{equation}\label{SS2}
f \circ T|_V = S \circ f|_V \ \ \ \text{ and } \ \ \ T \circ
f^{-1}|_U = f^{-1} \circ S|_U.
\end{equation}
Let $X = X_s^T \oplus X_u^T$ and $X = X_s^S \oplus X_u^S$ be the
hyperbolic splittings of $T$ and $S$, respectively. There are
constants $c \geq 1$ and $0 < \beta < 1$ such that
\begin{equation}\label{SS3}
\|T^nx\| \leq c \beta^n \|x\| \ \text{ and } \ \|S^ny\| \leq c
\beta^n \|y\| \ \text{ whenever } x \in X_s^T \text{ and } y \in
X_s^S.
\end{equation}
Given $x \in X_s^T$ and $y \in X_s^S$, let $n_x$ and $m_y$ be the
smallest positive integers such that $T^nx \in V$ and $S^my \in U$
whenever $n \geq n_x$ and $m \geq m_y$. We define $h_s(x) =
S^{-n_x}(f(T^{n_x}x))$ and $g_s(y) = T^{-m_y}(f^{-1}(S^{m_y}y))$. A
simple induction argument using (\ref{SS2}) shows that
\begin{equation}\label{SS4}
h_s(x) = S^{-n}(f(T^nx)) \ \text{ and } \ g_s(y) =
T^{-m}(f^{-1}(S^my)) \ \text{ if } n \geq n_x \text{ and } m \geq
m_y.
\end{equation}
Since $S^k(h_s(x)) = S^{-n_x}(f(T^{n_x+k}x)) \to 0$ and $T^k(g_s(y))
= T^{-m_y}(f^{-1}(S^{m_y+k}y)) \to 0$ as $k \to \infty$, we have
that $h_s(x) \in X_s^S$ and $g_s(y) \in X_s^T$. Hence, we have maps
$h_s : X_s^T \to X_s^S$ and $g_s : X_s^S \to X_s^T$. If $x \in
X_s^T$, $y = h_s(x) \in X_s^S$ and $n \geq \max\{n_x,m_y\}$, then
$g_s(h_s(x)) = T^{-n}(f^{-1}(S^n(S^{-n}(f(T^nx))))) = x$.
Analogously, $h_s(g_s(y)) = y$ for all $y \in X_s^S$. Thus, the maps
$h_s$ and $g_s$ are inverses of each other. Let us now prove that
$h_s$ is continuous. Let $(x_k)_{k \in \N}$ be a sequence in $X_s^T$
converging to a point $x \in X_s^T$. By (\ref{SS3}),
$$
\|T^nx - T^nx_k\| \leq c \beta^{n_x} \|x - x_k\| \ \text{ whenever }
k \in \N \text{ and } n \geq n_x.
$$
Since the sequence $(T^nx)_{n \geq n_x}$ lies in the open set $V$
and does not accumulate on its boun\-dary (since it converges to
zero), we see that there exists $k_0 \in \N$ such that $T^nx_k \in
V$ for all $k \geq k_0$ and $n \geq n_x$. Hence, by (\ref{SS4}),
$h_s(x_k) = S^{-n_x}(f(T^{n_x}x_k))$ for all $k \geq k_0$, which
clearly implies that $h_s(x_k) \to h_s(x)$ as $k \to \infty$. An
analogous argument shows that $g_s$ is also continuous. Thus, $h_s :
X_s^T \to X_s^S$ is a homeomorphism. By (\ref{SS2}) and (\ref{SS4}),
$h_s(Tx) = S^{-n_x}(f(T^{n_x+1}x)) = S^{-n_x+1}(f(T^{n_x}x)) =
S(h_s(x))$ for all $x \in X_s^T$, that is,
$$
h_s \circ T|_{X_s^T} = S|_{X_s^S} \circ h_s.
$$
A similar argument produces a homeomorphism $h_u : X_u^T \to X_u^S$
such that
$$
h_u \circ T|_{X_u^T} = S|_{X_u^S} \circ h_u.
$$
Finally, define $h : X \to X$ by $h(x) = h_s(x_s)+h_u(x_u)$ if $x =
x_s+x_u$ with $x_s \in X_s^T$ and $x_u \in X_u^T$. Then $h$ is a
homeomophism satisfying (\ref{SS1}).
\end{proof}

%%%%%%%%%%%%%%%%%%%%%%%%%%%%%%%%%%%%%%%%%%%%%%%%%%%%%%%%%%%%%%%%%%%%%%%%%%%%%

\end{document}